\documentclass[11pt]{amsart}
\usepackage{mathrsfs,latexsym,amsfonts,amssymb}
\newtheorem{theorem}{Theorem}[section]
\newtheorem{lemma}[theorem]{Lemma}
\newtheorem{corollary}[theorem]{Corollary}
\newtheorem{question}[theorem]{Question}

\theoremstyle{definition}
\newtheorem{definition}[theorem]{Definition}
\newtheorem{proposition}[theorem]{Proposition}
\theoremstyle{remark}
\newtheorem{remark}[theorem]{Remark}

\begin{document}

\title[Submetrizability of strongly topological gyrogroups]
{Submetrizability of strongly topological gyrogroups}

\author{Meng Bao}
\address{(Meng Bao): 1. School of mathematics and statistics,
Minnan Normal University, Zhangzhou 363000, P. R. China; 2. College of Mathematics, Sichuan University, Chengdu 610064, P. R. China}
\email{mengbao95213@163.com}

\author{Fucai Lin*}
\address{(Fucai Lin): School of mathematics and statistics,
Minnan Normal University, Zhangzhou 363000, P. R. China}
\email{linfucai2008@aliyun.com; linfucai@mnnu.edu.cn}

\thanks{The authors are supported by the Key Program of the Natural Science Foundation of Fujian Province (No: 2020J02043), the NSFC (Nos. 11571158), the Program for New Century Excellent Talents in Fujian Province University, the Lab of Granular Computing, the Institute of Meteorological Big Data-Digital Fujian and Fujian Key Laboratory of Data Science and Statistics.\\
*corresponding author}

\keywords{strongly topological gyrogroups; completely regular; countable pseudocharacter; submetrizable; admissible $L$-subgyrogroup.}
\subjclass[2020]{Primary 54H11; 22A05; secondary 54A25; 54B15; 54E35.}

\begin{abstract}
Topological gyrogroups, with a weaker algebraic structure without associative law, have been investigated recently. We prove that each $T_{0}$-strongly topological gyrogroup is completely regular. We also prove that every $T_{0}$-strongly topological gyrogroup with a countable pseudocharacter admits a weaker metrizable topology. Finally, we prove that the left coset space $G/H$ admits a weaker metrizable topology if $H$ is an admissible $L$-subgyrogroup of a $T_{0}$-strongly topological gyrogroup $G$.
\end{abstract}

\maketitle

\section{Introduction}

The gyrogroup was firstly posed by A.A. Ungar in the study of $c$-ball of relativistically admissible velocities with Einstein velocity addition \cite{UA}.
The Einstein velocity addition $\oplus_{E}$ in the $c$-ball is given by the following equation
$$\mathbf{u}\oplus_{E}\mathbf{v}=\frac{1}{1+\frac{\langle \mathbf{u}, \mathbf{v}\rangle}{c^{2}}}\left\{\mathbf{u}+\frac{1}{\gamma_{\mathbf{u}}}\mathbf{v}+\frac{1}{c^{2}}\frac{\gamma_{\mathbf{u}}}{1+\gamma_{\mathbf{u}}}\langle\mathbf{u}, \mathbf{v}\rangle\mathbf{u}\right\},$$where $\mathbf{u}, \mathbf{v}\in\mathbb{R}_{c}^{3}=\{\mathbf{v}\in\mathbb{R}^{3}: \parallel\mathbf{v}\parallel<c\}$ and $\gamma_{\mathbf{u}}$ is the Lorentz factor given by $$\gamma_{\mathbf{u}}=\frac{1}{\sqrt{1-\frac{\parallel\mathbf{u}\parallel^{2}}{c^{2}}}}.$$ The system $(\mathbb{R}_{c}^{3},
\oplus_{E})$ does not form a group since $\oplus_{E}$ is neither associative nor commutative.
Then, this topic has been studied by many authors during the past few years and they have achieved a large plenty of results about gyrogroups, see \cite{FM,FM1,LF,LF1,LF2,SL,ST,ST1,ST2,UA2002}. In particular, W. Atiponrat \cite{AW} defined topological gyrogroups and proved that the separations of $T_{0}$ and $T_{3}$ are equivalent in topological gyrogroups; Z. Cai, S. Lin and W. He in \cite{CZ} proved that every topological gyrogroup is a rectifiable space. T. Suksumran \cite{ST} defined the normal subgyrogroup of a gyrogroup and showed that the quotient space with the specific operation is a gyrogroup. In 2019, M. Bao and F. Lin \cite{BL} defined the strongly topological gyrogroups and proved that every feathered strongly topological gyrogroup is paracompact, which implies that every feathered strongly topological gyrogroup is a $D$-space. Indeed, M\"{o}bius gyrogroups, Einstein gyrogroups and Proper velocity gyrogroups with the usual topologies, that were studied e.g. in \cite{FM, FM1,UA}, are all strongly topological gyrogroups. In 2020, W. Atiponrat and R. Maungchang \cite{AW1} showed that every regular paratopological gyrogroup equipped with some mild conditions is completely regular.

In this paper, we continue to study the strongly topological gyrogroups. We mainly show that $T_{0}$ strongly topological gyrogroups satisfy the conditions in \cite[Lemma 3.6]{AW1}, and it follows that every $T_{0}$ strongly topological gyrogroup is completely regular. What's more, every strongly topological gyrogroup with a countable pseudocharacter is submetrizable. Finally, we prove that if $H$ is an admissible $L$-subgyrogroup of a strongly topological gyrogroup $G$, then the left coset space $G/H$ is submetrizable. These results extend some well known results of topological groups.

\smallskip
\section{Preliminaries}
In this section, we introduce the necessary notations, terminologies and some facts about topological gyrogroups.

Throughout this paper, all topological spaces are assumed to be
$T_{0}$, unless otherwise is explicitly stated. Let $\mathbb{N}$ be the set of all positive integers and $\omega$ the first infinite ordinal. Let $X$ be a topological space $X$ and $A \subseteq X$ be a subset of $X$.
  The {\it closure} of $A$ in $X$ is denoted by $\overline{A}$ and the
  {\it interior} of $A$ in $X$ is denoted by $\mbox{Int}(A)$. The readers may consult \cite{AA, E, linbook} for notation and terminology not explicitly given here.

\begin{definition}\cite{AW}
Let $G$ be a nonempty set, and let $\oplus: G\times G\rightarrow G$ be a binary operation on $G$. Then the pair $(G, \oplus)$ is called a {\it groupoid}. A function $f$ from a groupoid $(G_{1}, \oplus_{1})$ to a groupoid $(G_{2}, \oplus_{2})$ is called a {\it groupoid homomorphism} if $f(x\oplus_{1}y)=f(x)\oplus_{2} f(y)$ for any elements $x, y\in G_{1}$. Furthermore, a bijective groupoid homomorphism from a groupoid $(G, \oplus)$ to itself will be called a {\it groupoid automorphism}. We write $\mbox{Aut}(G, \oplus)$ for the set of all automorphisms of a groupoid $(G, \oplus)$.
\end{definition}

\begin{definition}\cite{UA}
Let $(G, \oplus)$ be a groupoid. The system $(G,\oplus)$ is called a {\it gyrogroup}, if its binary operation satisfies the following conditions:

\smallskip
(G1) There exists a unique identity element $0\in G$ such that $0\oplus a=a=a\oplus0$ for all $a\in G$.

\smallskip
(G2) For each $x\in G$, there exists a unique inverse element $\ominus x\in G$ such that $\ominus x \oplus x=0=x\oplus (\ominus x)$.

\smallskip
(G3) For all $x, y\in G$, there exists $\mbox{gyr}[x, y]\in \mbox{Aut}(G, \oplus)$ with the property that $x\oplus (y\oplus z)=(x\oplus y)\oplus \mbox{gyr}[x, y](z)$ for all $z\in G$.

\smallskip
(G4) For any $x, y\in G$, $\mbox{gyr}[x\oplus y, y]=\mbox{gyr}[x, y]$.
\end{definition}

Notice that a group is a gyrogroup $(G,\oplus)$ such that $\mbox{gyr}[x,y]$ is the identity function for all $x, y\in G$. The definition of a subgyrogroup is as follows.

\begin{definition}\cite{ST}
Let $(G,\oplus)$ be a gyrogroup. A nonempty subset $H$ of $G$ is called a {\it subgyrogroup}, denoted
by $H\leq G$, if the following statements hold:

\smallskip
(i) The restriction $\oplus| _{H\times H}$ is a binary operation on $H$, i.e. $(H, \oplus| _{H\times H})$ is a groupoid.

\smallskip
(ii) For any $x, y\in H$, the restriction of $\mbox{gyr}[x, y]$ to $H$, $\mbox{gyr}[x, y]|_{H}$ : $H\rightarrow \mbox{gyr}[x, y](H)$, is a bijective homomorphism.

\smallskip
(iii) $(H, \oplus|_{H\times H})$ is a gyrogroup.

\smallskip
Furthermore, a subgyrogroup $H$ of $G$ is said to be an {\it $L$-subgyrogroup} \cite{ST}, denoted
by $H\leq_{L} G$, if $\mbox{gyr}[a, h](H)=H$ for all $a\in G$ and $h\in H$.

\end{definition}

\begin{definition}\cite{AW}
A triple $(G, \tau, \oplus)$ is called a {\it topological gyrogroup} if the following statements hold:

\smallskip
(1) $(G, \tau)$ is a topological space.

\smallskip
(2) $(G, \oplus)$ is a gyrogroup.

\smallskip
(3) The binary operation $\oplus: G\times G\rightarrow G$ is jointly continuous while $G\times G$ is endowed with the product topology, and the operation of taking the inverse $\ominus (\cdot): G\rightarrow G$, i.e. $x\rightarrow \ominus x$, is also continuous.
\end{definition}

\begin{definition}\cite{BL}
Let $G$ be a topological gyrogroup. We say that $G$ is a {\it strongly topological gyrogroup} if there exists a neighborhood base $\mathscr U$ of $0$ such that, for every $U\in \mathscr U$, $\mbox{gyr}[x, y](U)=U$ for any $x, y\in G$. For convenience, we say that $G$ is a strongly topological gyrogroup with neighborhood base $\mathscr U$ of $0$. Clearly, we may assume that $U$ is symmetric for each $U\in\mathscr U$.
\end{definition}

\begin{remark}
By \cite[Example 3.1]{BL}, there exists a strongly topological gyrogroup which is not a topological group. What's more, the authors gave an example in \cite{BL} to show that there exists a strongly topological gyrogroup which has an infinite $L$-subgyrogroup, see \cite[Example 3.2]{BL}.
\end{remark}

\begin{definition}\cite{E}
A continuous mapping $f: X\rightarrow Y$ is called {\it closed} (resp. {\it open}) if for every closed (resp. open) set $A\subset X$ the image $f(A)$ is closed (resp. {\it open}) in $Y$.
\end{definition}

\begin{definition}\cite{E}
Let $f$ be a function from $X$ to $Y$. The {\it image} of the set $A\subset X$ under $f$ is the set $$f(A)=\{y\in Y: y=f(x)\ \mbox{for}\ \mbox{some}~~x\in A\},$$ and the {\it inverse image} of the set $B\subset Y$ under $f$ is the set $$f^{-1}(B)=\{x\in X: f(x)\in B\};$$ inverse images of one-point sets under $f$ are called {\it fibers} of $f$.
\end{definition}

We recall the following concept of the coset space of a topological gyrogroup.

Let $(G, \tau, \oplus)$ be a topological gyrogroup and $H$ an $L$-subgyrogroup of $G$. It follows from \cite[Theorem 20]{ST} that $G/H=\{a\oplus H:a\in G\}$ is a partition of $G$. We denote by $\pi$ the mapping $a\mapsto a\oplus H$ from $G$ onto $G/H$. Clearly, for each $a\in G$, we have $\pi^{-1}(\pi(a))=a\oplus H$ (that is, the fiber of the point $\pi(a)$). Indeed, for any $a\in G$ and $h\in H$,
\begin{eqnarray}
(a\oplus h)\oplus H&=&a\oplus (h\oplus gyr[h,a](H))\nonumber\\
&=&a\oplus (h\oplus gyr^{-1}[a,h](H))\nonumber\\
&=&a\oplus (h\oplus H)\nonumber\\
&=&a\oplus H\nonumber
\end{eqnarray}

Denote by $\tau (G)$ the topology of $G$. In the set $G/H$, we define a family $\tau (G/H)$ of subsets as follows: $$\tau (G/H)=\{O\subset G/H: \pi^{-1}(O)\in \tau (G)\}.$$

Finally, we give some facts about gyrogroups and topological gyrogroups, which are important in our proofs.

\begin{proposition}\cite{ST}.
Let $(G, \oplus)$ be a gyrogroup, and let $H$ be a nonempty subset of $G$. Then $H$ is a subgyrogroup if and only if the following statements are true:

\smallskip
(1) For any $x\in H$, $\ominus x\in H$.

\smallskip
(2) For any $x, y\in H$, $x\oplus y\in H$.
\end{proposition}

\begin{proposition}\cite{UA}\label{a}.
Let $(G, \oplus)$ be a gyrogroup. Then for any $x, y, z\in G$, we obtain the following:

\smallskip
(1) $(\ominus x)\oplus (x\oplus y)=y$.

\smallskip
(2) $(x\oplus (\ominus y))\oplus \mbox{gyr}[x, \ominus y](y)=x$.

\smallskip
(3) $(x\oplus \mbox{gyr}[x, y](\ominus y))\oplus y=x$.

\smallskip
(4) $\mbox{gyr}[x, y](z)=\ominus (x\oplus y)\oplus (x\oplus (y\oplus z))$.

\smallskip
(5) $(\ominus x\oplus y)\oplus \mbox{gyr}[\ominus x, b](\ominus y\oplus z)=\ominus x\oplus z$.

\smallskip
(6) $\ominus( x\oplus y)=\mbox{gyr}[x, y]( y\ominus x)$.

\smallskip
(7) $\ominus( x\oplus y)=\mbox{gyr}[x, y]( y\ominus x)$.
\end{proposition}

\begin{proposition}\cite{AW}\label{m}
Let $(G, \tau, \oplus)$ be a topological gyrogroup, and let $U$ be a neighborhood of the identity $0$. Then the following three statements hold:
\begin{enumerate}
\item There is an open symmetric neighborhood $V$ of $0$ in $G$ such that $V\subset U$ and $V\oplus V\subset U$.

\smallskip
\item There is an open neighborhood $V$ of $0$ such that $\ominus V\subset U$.

\smallskip
\item If $A$ is a subset of $G$, $\overline{A}\subset W\oplus A$ for any neighborhood $W$ of the identity $0$.
\end{enumerate}
\end{proposition}

\begin{proposition}\label{l}\cite{AW}
Let $(G,\tau ,\oplus)$ be a topological gyrogroup, and let $A$ be a subgyrogroup of $G$. Then the followings are true:
\begin{enumerate}
\item If $\mbox{Int}(A)\not =\emptyset$, then $A$ is open.

\smallskip
\item If $A$ is open, then it is also closed.

\smallskip
\item $\overline{A}$ is a subgyrogroup of $G$.
\end{enumerate}
\end{proposition}

\section{Complete regularity of strongly topological gyrogroups}
In this section, we mainly prove that every $T_{0}$-strongly topological gyrogroup is completely regular.
In \cite{AW}, W. Atiponrat proved that the separation axioms $T_{0}$ and $T_{3}$ are equivalent for a topological gyrogroup and posed the following question.

\begin{question}\cite{AW}
Is every Hausdorff topological gyrogroup completely regular?
\end{question}

It is clear that every strongly topological gyrogroup is a topological gyrogroup. Hence, it is natural to pose the following question.

\begin{question}\label{q1}
Is every Hausdorff strongly topological gyrogroup completely regular?
\end{question}

Next we prove that each $T_{0}$-strongly topological gyrogroup is completely regular, which gives an affirmative answer to Question \ref{q1}, see Theorem \ref{t1}. In \cite{AW1}, the authors defined a micro-associative paratopological gyrogroup and showed that a micro-associative regular paratopological gyrogroup with additional conditions is completely regular. We recall this concept as follows.

\begin{definition}\cite{AW1}
We call a paratopological gyrogroup $G$ is {\it micro-associative} if for any neighborhood $U\subset G$ of $0$, there are neighborhoods $W$ and $V$ of $0$ such that $W\subset V\subset U$ and $a\oplus (b\oplus V)=(a\oplus b)\oplus V$ for any $a,b\in W$.
\end{definition}

In \cite{CZ}, they introduced the concept of paratopological gyrogroup such that each topological gyrogroup is a paratopological gyrogroup.

\begin{lemma}\cite{AW1}\label{l1}
Let $(G,\tau ,\oplus)$ be a micro-associative Hausdorff paratopological gyrogroup such that for any neighborhoods $A$, $B$ of $0$, $$\overline{A}\subset \mbox{Int}(\overline{B\oplus A}).$$ Then $G$ is completely regular.
\end{lemma}

Next, we show that every $T_{0}$-strongly topological gyrogroup $G$ is a micro-associative Hausdorff topological gyrogroup and satisfies the additional condition in Lemma~\ref{l1}, thus $G$ is completely regular.

\begin{theorem}\label{t1}
Every $T_{0}$-strongly topological gyrogroup is completely regular.
\end{theorem}

\begin{proof}
Let $G$ be a $T_{0}$-strongly topological gyrogroup with a symmetric neighborhood base $\mathscr U$ of $0$. Therefore, for every $U\in \mathscr U$, we have $\mbox{gyr}[x, y](U)=U$ for any $x,y\in G$. For any neighborhood $O\subset G$ of $0$, there are $U_{1}, U_{2}\in \mathscr U$ such that $U_{1}\subset U_{2}\subset O$. Moreover, for every $a, b\in U_{1}$, $$a\oplus (b\oplus U_{2})=(a\oplus b)\oplus \mbox{gyr}[a, b](U_{2})=(a\oplus b)\oplus U_{2}.$$ Therefore, a strongly topological gyrogroup is a micro-associative topological gyrogroup. Let $U, V$ be any neighborhoods of $0$ in $G$. It follows from \cite[Lemma 9]{AW} that $\overline{U}\subset V\oplus U\subset \mbox{Int}(\overline{V\oplus U})$. Hence, by Lemma \ref{l1}, we have that a $T_{0}$-strongly topological gyrogroup is completely regular.
\end{proof}

\section{Submetrizability of a strongly topological gyrogroups}
In this section, we mainly discuss the submetrizabilities of a strongly topological gyrogroup. A space $X$ is {\it submetrizable} if there exists a continuous one-to-one mapping of $X$ onto a metrizable space. Moreover, a space $X$ is said to be of {\it countable pseudocharacter} if, for every $x\in X$, there exists a sequence of open subsets $\{U_{n}(x)\}_{n\in\mathbb{N}}$ such that $\{x\}=\bigcap_{n\in\mathbb{N}}U_{n}(x)$. It is well known that every topological group with a countable pseudocharacter is submetrizable. In \cite{BL}, the authors posed the following question.

\begin{question}\label{wt1}
Let $(G, \tau, \oplus)$ be a topological gyrogroup with a countable pseudocharacter. Is $G$ submetrizable? What if the topological gyrogroup is a strongly topological gyrogroup?
\end{question}

Next we prove that a strongly topological gyrogroup with a countable pseudocharacter is submetrizable, which gives an affirmative answer to Question \ref{wt1} when the topological gyrogroup is a strongly topological gyrogroup, see Theorem \ref{dl3}. First, we recall an important lemma in \cite{BL}.

\begin{lemma}\label{s}\cite{BL}
Let $G$ be a strongly topological gyrogroup with the symmetric neighborhood base $\mathscr{U}$ at $0$, and let $\{U_{n}: n\in\mathbb{N}\}$ and $\{V(m/2^{n}): n, m\in\mathbb{N}\}$ be two sequences of open neighborhoods satisfying the following conditions (1)-(5):

\smallskip
(1) $U_{n}\in\mathscr{U}$ for each $n\in \mathbb{N}$.

\smallskip
(2) $U_{n+1}\oplus U_{n+1}\subset U_{n}$, for each $n\in \mathbb{N}$.

\smallskip
(3) $V(1)=U_{0}$;

\smallskip
(4) For any $n\geq 1$, put $$V(1/2^{n})=U_{n}, V(2m/2^{n})=V(m/2^{n-1})$$ for $m=1,...,2^{n-1}$, and $$V((2m+1)/2^{n})=U_{n}\oplus V(m/2^{n-1})=V(1/2^{n})\oplus V(m/2^{n-1})$$ for each $m=1,...,2^{n-1}-1$;

\smallskip
(5) $V(m/2^{n})=G$ when $m>2^{n}$;

\smallskip
Then there exists a prenorm $N$ on $G$ that satisfies the following conditions:

\smallskip
(a) for any fixed $x, y\in G$, we have $N(\mbox{gyr}[x,y](z))=N(z)$ for any $z\in G$;

\smallskip
(b) for any $n\in \mathbb{N}$, $$\{x\in G: N(x)<1/2^{n}\}\subset U_{n}\subset\{x\in G: N(x)\leq 2/2^{n}\}.$$
\end{lemma}

\begin{theorem}\label{dl3}
Suppose that $G$ is a strongly topological gyrogroup with the symmetric neighborhood base $\mathscr{U}$ at $0$. If $G$ is of countable pseudocharacter, then $G$ is submetrizable.
\end{theorem}

\begin{proof}
Since the identity element $0$ is a $G_{\delta}$-point in $G$, there exists a sequence $\{W_{n}: n\in \mathbb{N}\}$ of open sets in $G$ such that $0=\bigcap _{n\in \mathbb{N}}W_{n}$. By induction, we obtain a sequence $\{U_{n}: n\in \mathbb{N}\}$ of symmetric open neighborhoods of $0$ such that $U_{n}\in \mathscr U$, $U_{n}\subset W_{n}$ and $U_{n+1}\oplus U_{n+1}\subset U_{n}$ for each $n\in \mathbb{N}$.

Apply Lemma \ref{s} to choose a continuous prenorm $N$ on $G$ which satisfies $$N(\mbox{gyr}[x, y](z))=N(z)$$ for any $x, y, z\in G$ and $$\{x\in G: N(x)<1/2^{n}\}\subset U_{n}\subset \{x\in G: N(x)\leq 2/2^{n}\},$$ for each integer $n\geq 0$.

Now, for arbitrary $x$ and $y$ in $G$, put $\varrho _{N}(x, y)=N(\ominus x\oplus y)+N(\ominus y\oplus x)$. Let us show that $\varrho _{N}$ is a metric on $G$.

\smallskip
(1) Clearly, $\varrho _{N}(x, y)=N(\ominus x\oplus y)+N(\ominus y\oplus x)\geq 0$, for every $x, y\in G$. At the same time, $\varrho _{N}(x, x)=N(0)+N(0)=0$, for each $x\in G$. Assume that $$\varrho _{N}(x, y)=N(\ominus x\oplus y)+N(\ominus y\oplus x)=0,$$ that is, $N(\ominus x\oplus y)=N(\ominus y\oplus x)=0$. Then, for each $n\in\mathbb{N}$, $$\ominus x\oplus y\in \{x\in G: N(x)<1/2^{n}\}\subset U_{n}\subset W_{n}$$ and so is the $\ominus y\oplus x$. Since $\{0\}=\bigcap _{n\in \mathbb{N}}W_{n}$, it follows that $\ominus x\oplus y=0=\ominus y\oplus x$, that is, $x=y$.

\smallskip
(2) For every $x, y\in G$, $\varrho _{N}(y, x)=N(\ominus y\oplus x)+N(\ominus x\oplus y)=\varrho _{N}(x, y)$.

\smallskip
(3) For every $x, y, z\in G$, it follows from \cite[Theorem 2.11]{UA2005} that
\begin{eqnarray}
\varrho _{N}(x, y)&=&N(\ominus x\oplus y)+N(\ominus y\oplus x)\nonumber\\
&=&N((\ominus x\oplus z)\oplus \mbox{gyr}[\ominus x, z](\ominus z\oplus y))\nonumber\\
&&+N((\ominus y\oplus z)\oplus \mbox{gyr}[\ominus y, z](\ominus z\oplus x))\nonumber\\
&\leq&N(\ominus x\oplus z)+N(\mbox{gyr}[\ominus x, z](\ominus z\oplus y))\nonumber\\
&&+N(\ominus y\oplus z)+N(\mbox{gyr}[\ominus y, z](\ominus z\oplus x))\nonumber\\
&=&N(\ominus x\oplus z)+N(\ominus z\oplus y)+N(\ominus y\oplus z)+N(\ominus z\oplus x)\nonumber\\
&=&\varrho _{N}(x, z)+\varrho _{N}(z, y)\nonumber
\end{eqnarray}

Thus, $\varrho _{N}$ is a metric on $G$.
Clearly, the topology generated by the metric $\varrho _{N}$ of $G$ is weaker than the original topology. Therefore, $G$ is submetrizable.
\end{proof}

Next we consider the coset space of a topological gyrogroup, and prove that the coset space $G/H$ of topological gyrogroup $G$, where $H$ is an admissible $L$-subgyrogroup, is submetrizable.

A subgyrogroup $H$ of a topological gyrogroup $G$ is called {\it admissible} if there exists a sequence $\{U_{n}:n\in \mathbb{N}\}$ of open symmetric neighborhoods of the identity $0$ in $G$ such that $U_{n+1}\oplus (U_{n+1}\oplus U_{n+1})\subset U_{n}$ for each $n\in \mathbb{N}$ and $H=\bigcap _{n\in \mathbb{N}}U_{n}$. If $G$ is a strongly topological gyrogroup with a symmetric neighborhood base $\mathscr U$ at $0$ and each $U_{n}\in \mathscr U$, we say that the admissible topological subgyrogroup is generated from $\mathscr U$. It was claimed in \cite{BL2} that the admissible topological subgyrogroup generated
from $\mathscr U$ of a strongly topological gyrogroup is a closed L-subgyrogroup. The following proposition shows that if the admissible subgyrogroup $H$ is generated from $\mathscr U$, then the coset space $G/H$ is a homogenous space.

\begin{proposition}
Let $(G, \tau, \oplus)$ be a strongly topological gyrogroup with a symmetric neighborhood base $\mathscr U$ at $0$. If $H$ is an admissible subgyrogroup generated from $\mathscr U$, then the coset space $G/H$ is a homogenous space.
\end{proposition}

\begin{proof}
Since $H$ is an admissible subgyrogroup generated from $\mathscr U$, there exists a sequence $\{U_{n}: n\in \mathbb{N}\}\subset \mathscr U$ such that $U_{n+1}\oplus (U_{n+1}\oplus U_{n+1})\subset U_{n}$ for each $n\in \mathbb{N}$ and $H=\bigcap _{n\in \mathbb{N}}U_{n}$. Now we prove that the coset space $G/H$ is a homogenous space. Indeed, for any $a\in G$, define a mapping $h_{a}$ of
$G/H$ to itself by the rule $$h_{a}(x\oplus H)=(a\oplus x)\oplus H.$$ First, we prove that this definition is correct. Hence it suffices to prove that $(a\oplus x)\oplus H=(a\oplus y)\oplus H$ if $x\oplus H=y\oplus H$. Then it only need to prove $(a\oplus z)\oplus H=a\oplus (z\oplus H)$ for any $z\in G$. Take an arbitrary $z\in G$. Then
\begin{eqnarray}
(a\oplus z)\oplus H&=&a\oplus (z\oplus \mbox{gyr}[z, a](H))\nonumber\\
&=&a\oplus (x\oplus \mbox{gyr}[z, a](\bigcap _{n\in \mathbb{N}}U_{n}))\nonumber\\
&\subset &a\oplus (z\oplus \bigcap _{n\in \mathbb{N}}\mbox{gyr}[z, a](U_{n}))\nonumber\\
&=&a\oplus (z\oplus \bigcap _{n\in \mathbb{N}}U_{n})\nonumber\\
&=&a\oplus (z\oplus H).\nonumber
\end{eqnarray}
What's more, we have
\begin{eqnarray}
a\oplus (z\oplus H)&=&(a\oplus z)\oplus \mbox{gyr}[a, z](H)\nonumber\\
&=&(a\oplus z)\oplus \mbox{gyr}[a, z](\bigcap _{n\in \mathbb{N}}U_{n})\nonumber\\
&\subset &(a\oplus z)\oplus \bigcap _{n\in \mathbb{N}}\mbox{gyr}[a, z](U_{n})\nonumber\\
&=&(a\oplus z)\oplus \bigcap _{n\in \mathbb{N}}U_{n}\nonumber\\
&=&(a\oplus z)\oplus H.\nonumber
\end{eqnarray}

Therefore, $(a\oplus z)\oplus H=a\oplus (z\oplus H)$. Clearly, $h_{a}$ is a bijection. Next we prove that $h_{a}$
is a homeomorphism. This can be seen from the following argument.

Take any $x\oplus H\in G/H$ and any open basic neighbourhood $U$ of $0$ in $G$. Then it follows from \cite[Theorem 3.7]{BL} that $\pi((x\oplus U)\oplus H)$ is a basic neighbourhood of $x\oplus H$ in $G/H$. Similarly, the set $\pi(a\oplus ((x\oplus U)\oplus H))$ is a basic neighbourhood of
$a\oplus (x\oplus H)$ in $G/H$. Since $h_{a}(\pi((x\oplus U)\oplus H))=\pi(a\oplus ((x\oplus U)\oplus H))$, it easily verify that $h_{a}$ is a homeomorphism. Now, for any given $x\oplus H$ and $y\oplus H$ in
$G/H$, we can take $a=y\oplus \mbox{gyr}[y, x](\ominus x)$. We claim that $h_{a}(x\oplus H)=y\oplus H$. In fact, from the above proof, it follows that $\mbox{gyr}[y\oplus \mbox{gyr}[y, x](\ominus x), x](H)=H$. Since $H$ is an admissible subgyrogroup generated from $\mathscr U$,
\begin{eqnarray}
h_{a}(x\oplus H)&=&(y\oplus \mbox{gyr}[y, x](\ominus x))\oplus (x\oplus H)\nonumber\\
&=&((y\oplus \mbox{gyr}[y, x](\ominus x))\oplus x)\oplus \mbox{gyr}[y\oplus \mbox{gyr}[y, x](\ominus x), x](H)\nonumber\\
&=&y\oplus H.\nonumber
\end{eqnarray}
Hence, the quotient space $G/H$ is homogeneous.
\end{proof}

\begin{lemma}\label{3yl1}
Let $(G, \tau, \oplus)$ be a strongly topological gyrogroup with a symmetric neighborhood base $\mathscr U$ at $0$. Then:

\smallskip
(a) every admissible subgyrogroup $H$ of $G$ generated by $\mathscr U$ is closed in $G$ and the quotient space $G/H$ has countable pseudocharacter;

\smallskip
(b) every neighborhood of the identity $0$ in $G$ contains an admissible $L$-subgyrogroup;

\smallskip
(c) the intersection of countably many admissible $L$-subgyrogroups of $G$ is again an admissible $L$-subgyrogroup of $G$.
\end{lemma}

\begin{proof}
(a) Suppose that $H$ is an admissible subgyrogroup of $G$ generated by $\mathscr U$. Then there exists a sequence $\{U_{n}:n\in \mathbb{N}\}$ of open symmetric neighborhoods of $0$ in $G$ such that $U_{n}\in \mathscr U$ and $U_{n+1}\oplus (U_{n+1}\oplus U_{n+1})\subset U_{n}$ for each $n\in \mathbb{N}$ and $H=\bigcap _{n\in \mathbb{N}}U_{n}$. Since $\overline{U_{n+1}}\subset U_{n+1}\oplus U_{n+1}\subset U_{n+1}\oplus (U_{n+1}\oplus U_{n+1})\subset U_{n}$ for each $n\in \mathbb{N}$, the intersection of the sets $U_{n}$ coincides with the intersection of their closures. Hence, $H$ is closed in $G$.

Let $\pi: G\rightarrow G/H$ be the quotient mapping of $G$ onto the left coset space $G/H$. For every $n\in \mathbb{N}$, we have $$\pi^{-1}(\pi (U_{n+1}))=U_{n+1}\oplus H\subset U_{n+1}\oplus U_{n+1}\subset U_{n}.$$ Therefore, the set $P=\bigcap _{n\in \mathbb{N}}\pi (U_{n})$ satisfies $$\pi^{-1}(P)=\bigcap _{n\in \mathbb{N}}\pi^{-1}\pi(U_{n+1})\subset \bigcap _{n\in \mathbb{N}}U_{n}=H.$$ Hence $\pi(P)=\{0\}.$ For any $a\in G$, we have $$\pi(a)\in\bigcap_{n\in \mathbb{N}}(\pi(a)\oplus \pi (U_{n}))=\pi(a)\oplus \bigcap_{n\in \mathbb{N}}\pi(U_{n})=\{\pi(a)\},$$ that is, $\{\pi(a)\}=\bigcap_{n\in \mathbb{N}}(\pi(a)\oplus \pi (U_{n}))$. Therefore, $G/H$ has countable pseudocharacter.

\smallskip
(b) Let $U$ be an arbitrary neighborhood of $0$ in $G$. Define a sequence $\{U_{n}: n\in \mathbb{N}\}$ of open symmetric neighborhoods of $0$ in $G$ such that $U_{0}\subset U$, $U_{n}\in \mathscr U$ and $U_{n+1}\oplus (U_{n+1}\oplus U_{n+1})\subset U_{n}$ for each $n\in \mathbb{N}$. Then $H=\bigcap _{n\in \mathbb{N}} U_{n}$ is an admissible $L$-subgyrogroup of $G$ and $H\subset U_{0}\subset U$.

\smallskip
(c) Let $\{H_{n}: n\in \mathbb{N}\}$ be a sequence of admissible $L$-subgyrogroups of $G$. For every $n\in \mathbb{N}$, we can find a sequence $\{U_{n, k}: k\in \mathbb{N}\}$ of open symmetric neighborhoods of $0$ in $G$ such that $U_{n, k}\in \mathscr U$ and $U_{n, k+1}\oplus (U_{n, k+1}\oplus U_{n, k+1})\subset U_{n, k}$ for each $k\in \mathbb{N}$ and $H_{n}=\bigcap _{k\in \mathbb{N}}U_{n, k}$. Consider the sequence $\{V_{n}: n\in \mathbb{N}\}$, where $V_{n}=\bigcap _{i=0}^{n}U_{i, n}$ for each $n\in \mathbb{N}$. Clearly, every $V_{n}$ is an open symmetric neighborhood of $0$ in $G$. In addition, whenever $i\leq n$, $$V_{n+1}\oplus (V_{n+1}\oplus V_{n+1})\subset U_{i, n+1}\oplus (U_{i, n+1}\oplus U_{i, n+1})\subset U_{i, n}.$$ So, for all $n\in \mathbb{N}$, $$V_{n+1}\oplus (V_{n+1}\oplus V_{n+1})\subset \bigcap _{i=0}^{n}U_{i, n}=V_{n}.$$ What's more, for every $x, y\in G$,
\begin{eqnarray}
\mbox{gyr}[x, y](H)&=&\mbox{gyr}[x, y](\bigcap _{n\in \mathbb{N}}V_{n})\nonumber\\
&=&\mbox{gyr}[x, y](\bigcap _{n\in \mathbb{N}}\bigcap _{i=0}^{n}U_{i, n})\nonumber\\
&\subset &\bigcap _{n\in \mathbb{N}}\bigcap _{i=0}^{n}\mbox{gyr}[x, y](U_{i, n})\nonumber\\
&=&\bigcap _{n\in \mathbb{N}}\bigcap _{i=0}^{n}U_{i, n}\nonumber\\
&=&\bigcap _{n\in \mathbb{N}}V_{n}\nonumber\\
&=&H.\nonumber
\end{eqnarray}
By \cite[Proposition 2.6]{ST}, we have $\mbox{gyr}[x, y](H)=H$ for all $x, y\in G$. Moreover, $$H=\bigcap _{n\in \mathbb{N}}V_{n}=\bigcap _{n\in \mathbb{N}}\bigcap _{i=0}^{n}U_{i, n}=\bigcap _{i\in \mathbb{N}}\bigcap _{n=i}^{\infty }U_{i, n}=\bigcap _{i\in \mathbb{N}}H_{i}.$$ Therefore, $\bigcap _{i\in \mathbb{N}}H_{i}$ is an admissible $L$-subgyrogroup.
\end{proof}

A subset $B$ of a space $X$ is said to be
{\it bounded} in $X$ if every continuous real-valued function on $X$ is bounded
on $B$.

\begin{lemma}\cite{AA}\label{3yl2}
A subset $B$ of a Tychonoff space $X$ is bounded in $X$ iff for every locally finite family $\gamma$ of open sets in $X$, the set $B$ meets only finitely many elements of $\gamma$.
\end{lemma}

\begin{theorem}
Suppose that $(G, \tau, \oplus)$ is a strongly topological gyrogroup with a symmetric neighborhood base $\mathscr U$ at $0$. Suppose further that $B$ is a subset of $G$ such that the set $\pi _{H}(B)$ is bounded in the quotient space $G/H$ for every admissible $L$-subgyrogroup $H$ of $G$, where $\pi _{H}:G\rightarrow G/H$ is the quotient mapping. Then $B$ is bounded in $G$.
\end{theorem}

\begin{proof}
Suppose to the contrary that $B$ is not bounded in $G$. Then, by Lemma \ref{3yl2}, there exists a locally finite family $\gamma =\{V_{n}:n\in \mathbb{N}\}$ of open sets in $G$ such that $B$ meets every $V_{n}$. For every $n\in \mathbb{N}$, choose a point $x_{n}\in B\cap V_{n}$ and an open neighborhoof $W_{n}$ of the identity $0$ such that $W_{n}\in \mathscr U$ and $(x_{n}\oplus W_{n})\oplus W_{n}\subset V_{n}$. According to (b) and (c) of Lemma \ref{3yl1}, $W_{n}$ contains an admissible $L$-subgyrogroup $H_{n}$ of $G$, and $H=\bigcap _{n\in \mathbb{N}}H_{n}$ is also an admissible $L$-subgyrogroup of $G$. Let $\pi :G\rightarrow G/H$ be the quotient mapping. Then, for each $n\in \mathbb{N}$, $$\pi ^{-1}(\pi (x_{n}\oplus W_{n}))=(x_{n}\oplus W_{n})\oplus H\subset (x_{n}\oplus W_{n})\oplus H_{n}\subset (x_{n}\oplus W_{n})\oplus W_{n}\subset V_{n}.$$ By \cite[Theorem 3.7]{BL}, the mapping $\pi$ is open. Then since $\gamma$ is locally finite in $G$, it follows that the family $\{\pi (x_{n}\oplus W_{n}):n\in \mathbb{N}\}$ of open sets is locally finite in $G/H$. It is clear that $\pi (B)$ meets every element of this family, so that $\pi (B)$ is not bounded in $G/H$ by Lemma \ref{3yl2}.
\end{proof}

\begin{theorem}
Suppose that $G$ is a strongly topological gyrogroup with the symmetric neighborhood base $\mathscr{U}$ at $0$, $H$ is an admissible $L$-subgyrogroup of $G$ generated from $\mathscr U$, then the left coset space $G/H$ is submetrizable.
\end{theorem}

\begin{proof}
Let $\{U_{n}:n\in \mathbb{N}\}$ be a sequence of symmetric open neighborhoods of the identity $0$ in $G$ satisfying $U_{n}\in \mathscr U$ and $U_{n+1}\oplus (U_{n+1}\oplus U_{n+1})\subset U_{n}$, for each $n\in \mathbb{N}$, and such that $H=\bigcap _{n\in \mathbb{N}}U_{n}$. By Lemma \ref{s}, there exists a continuous prenorm $N$ on $G$ which satisfies $$N(\mbox{gyr}[x,y](z))=N(z)$$ for any $x, y, z\in G$ and $$\{x\in G: N(x)<1/2^{n}\}\subset U_{n}\subset\{x\in G: N(x)\leq 2/2^{n}\},$$ for any $n\in \mathbb{N}$.

First, we show that $N(x)=0$ if and only if $x\in H$. If $N(x)=0$, then $$x\in \bigcap _{n\in \mathbb{N}}\{x\in G:N(x)<1/2^{n}\}\subset \bigcap _{n\in \mathbb{N}}U_{n}=H.$$ On the other hand, if $x\in H$, since $$H=\bigcap _{n\in \mathbb{N}}U_{n}\subset \bigcap _{n\in \mathbb{N}}\{x\in G:N(x)\leq 2/2^{n}\},$$ we have that $N(x)=0$.

We claim that $N(x\oplus h)=N(x)$ for every $x\in G$ and $h\in H$. Indeed, for every $x\in G$ and $h\in H$, $N(x\oplus h)\leq N(x)+N(h)=N(x)+0=N(x)$. Moreover, by the definition of $N$, we observe that $N(\mbox{gyr}[x,y](z))=N(z)$ for every $x,y,z\in G$. Since $H$ is a $L$-subgyrogroup, it follows that
\begin{eqnarray}
N(x)&=&N((x\oplus h)\oplus \mbox{gyr}[x,h](\ominus h))\nonumber\\
&\leq&N(x\oplus h)+N(\mbox{gyr}[x,h](\ominus h))\nonumber\\
&=&N(x\oplus h)+N(\ominus h)\nonumber\\
&=&N(x\oplus h).\nonumber
\end{eqnarray}
Therefore, $N(x\oplus h)=N(x)$ for every $x\in G$ and $h\in H$.

Now define a function $d$ from $G\times G$ to $\mathbb{R}$ by $d(x,y)=|N(x)-N(y)|$ for all $x,y\in G$. Obviously, $d$ is continuous. We show that $d$ is a pseudometric.

\smallskip
(1) For any $x, y\in G$, if $x=y$, then $d(x, y)=|N(x)-N(y)|=0$.

\smallskip
(2) For any $x, y\in G$, $d(y, x)=|N(y)-N(x)|=|N(x)-N(y)|=d(x, y)$.

\smallskip
(3) For any $x, y, z\in G$, we have
\begin{eqnarray}
d(x, y)&=&|N(x)-N(y)|\nonumber\\
&=&|N(x)-N(z)+N(z)-N(y)|\nonumber\\
&\leq&|N(x)-N(z)|+|N(z)-N(y)|\nonumber\\
&=&d(x, z)+d(z, y).\nonumber
\end{eqnarray}

If $x'\in x\oplus H$ and $y'\in y\oplus H$, then there exist $h_{1},h_{2}\in H$ such that $x'=x\oplus h_{1}$ and $y'=y\oplus h_{2}$, then $$d(x', y')=|N(x\oplus h_{1})-N(y\oplus h_{2})|=|N(x)-N(y)|=d(x, y).$$ This enables us to define a function $\varrho $ on $G/H\times G/H$ by $$\varrho (\pi _{H}(x),\pi _{H}(y))=d(\ominus x\oplus y, 0)+d(\ominus y\oplus x, 0)$$ for any $x, y\in G$.

It is obvious that $\varrho $ is continuous, and we verify that $\varrho $ is a metric on $G/H$.

\smallskip
(1) Obviously, for any $x, y\in G$, then
\begin{eqnarray}
\varrho (\pi _{H}(x),\pi _{H}(y))=0&\Leftrightarrow&d(\ominus x\oplus y, 0)=d(\ominus y\oplus x, 0)=0\nonumber\\
&\Leftrightarrow&N(\ominus x\oplus y)=N(\ominus y\oplus x)=0\nonumber\\
&\Leftrightarrow&\ominus x\oplus y\in H\ \mbox{and}\ \ominus y\oplus x\in H\nonumber\\
&\Leftrightarrow&y\in x\oplus H\ \mbox{and}\ x\in y\oplus H\nonumber\\
&\Leftrightarrow&\pi _{H}(x)=\pi _{H}(y).\nonumber
\end{eqnarray}

\smallskip
(2) For every $x,y\in G$, it is obvious that $\varrho (\pi _{H}(y), \pi _{H}(x))=\varrho (\pi _{H}(x),\pi _{H}(y))$.

\smallskip
(3) For every $x, y, z\in G$, it follows from \cite[Theorem 2.11]{UA2005} that
\begin{eqnarray}
\varrho (\pi _{H}(x),\pi _{H}(y))&=&N(\ominus x\oplus y)+N(\ominus y\oplus x)\nonumber\\
&=&N((\ominus x\oplus z)\oplus \mbox{gyr}[\ominus x,z](\ominus z\oplus y))\nonumber\\
&&+N((\ominus y\oplus z)\oplus \mbox{gyr}[\ominus y,z](\ominus z\oplus x))\nonumber\\
&\leq&N(\ominus x\oplus z)+N(\mbox{gyr}[\ominus x,z](\ominus z\oplus y))\nonumber\\
&&+N(\ominus y\oplus z)+N(\mbox{gyr}[\ominus y,z](\ominus z\oplus x))\nonumber\\
&=&N(\ominus x\oplus z)+N(\ominus z\oplus y)+N(\ominus y\oplus z)+N(\ominus z\oplus x)\nonumber\\
&=&d(\ominus x\oplus z, 0)+d(\ominus z\oplus x, 0)+d(\ominus z\oplus y, 0)+d(\ominus y\oplus z, 0)\nonumber\\
&=&\varrho (\pi _{H}(x),\pi _{H}(z))+\varrho (\pi _{H}(z),\pi _{H}(y)).\nonumber
\end{eqnarray}

Given any points $x\in G$, $y\in G/H$ and any $\varepsilon >0$, we define open balls, $$B(x, \varepsilon)=\{x'\in G: d(x',x)<\varepsilon\}$$ and $$B^{*}(y, \varepsilon)=\{y'\in G/H: \varrho (y',y)<\varepsilon\}$$ in $G$ and $G/H$, respectively. Obviously, if $x\in G$ and $y=\pi _{H}(x)$, then we have $B(x, \varepsilon)=\pi ^{-1}_{H}(B^{*}(y, \varepsilon))$. Therefore, the topology generated by $\varrho$ on $G/H$ is coarser than the quotient topology on $G/H$ and the space $G/H$ is submetrizable.
\end{proof}

Since every completely regular submetrizable space is Dieudonn\'{e} complete, we have the following corollary.

\begin{corollary}
Every $T_{0}$-strongly topological gyrogroup with a countable pseudocharacter is Dieudonn\'{e} complete.
\end{corollary}


\begin{thebibliography}{33}

\bibitem{AA} A.V. Arhangel' ski\v\i, M. Tkachenko, {\it Topological Groups and Related Structures}, Atlantis Press and World Sci., 2008.

\bibitem{AW} W. Atiponrat, {\it Topological gyrogroups: generalization of topological groups}, Topol. Appl., 224 (2017) 73--82.

\bibitem{AW1} W. Atiponrat, R. Maungchang, {\it Complete regularity of paratopological gyrogroups}, Topol. Appl., 270 (2020) 106951.

\bibitem{BL} M. Bao, F. Lin, {\it Feathered gyrogroups and gyrogroups with countable pseudocharacter}, Filomat, 33(16)(2019) 5113-5124.

\bibitem{BL2} M. Bao, F. Lin, {\it Quotient with respect to admissible L-subgyrogroups}, https://arxiv. org/abs/2003.08843.

\bibitem{CZ} Z. Cai, S. Lin, W. He, {\it A note on Paratopological Loops}, Bull. Malay. Math. Sci. Soc., 42(5)(2019) 2535--2547.

\bibitem{E} R. Engelking, {\it General Topology}(revised and completed edition), Heldermann Verlag,
Berlin, 1989.

\bibitem{FM} M. Ferreira, {\it Factorizations of M\"{o}bius gyrogroups}, Adv. Appl. Clifford Algebras, 19(2009)303--323.

\bibitem{FM1} M. Ferreira, G. Ren, {\it M\"{o}bius gyrogroups: A Clifford algebra approach}, J. Algebra, 328(2011)230-253.

\bibitem{LF} F. Lin. R. Shen, {\it On rectifiable spaces and paratopological groups}, Topol. Appl., 158(2011)597--610.

\bibitem{LF1} F. Lin. C. Liu, S. Lin, {\it A note on rectifiable spaces}, Topol. Appl., 159(2012)2090--2101.

\bibitem{LF2} F. Lin, {\it Compactly generated rectifiable spaces or paratopological groups}, Math. Commun., 18(2013)417--427.

\bibitem{linbook} S. Lin, Z. Yun, {\it Generalized Metric Spaces and Mappings}, Science Press, Atlantis Press, 2017.

\bibitem{SL} L.V. Sabinin, L.L. Sabinin, L.V. Sbitneva, {\it On the notion of gyrogroup}, Aequ. Math., 56(1998)11--17.

\bibitem{ST} T. Suksumran, K. Wiboonton, {\it Isomorphism theorems for gyrogroups and $L$-subgyrogroups}, J. Geom. Symmetry Phys., 37(2015)67--83.

\bibitem{ST1} T. Suksumran, {\it Essays in mathematics and its applications: in honor of Vladimir Arnold, in: P.M. Pardalos, T.M. Rassias (Eds.), The Algebra of Gyrogroups: Cayley's Theorem, Lagrange's Theorem, and Isomorphism Theorems}, Springer, 2016, pp. 369--437.

\bibitem{ST2} T. Suksumran, {\it Special subgroups of gyrogroups: commutators, nuclei and radical}, Math. Interdiscip. Res, 1(2016)53--68.

\bibitem{UA} A.A. Ungar, {\it Analytic Hyperbolic Geometry and Albert Einstein's Special Theory of Relativity}, World Scientific, Hackensack, New Jersey, 2008.

\bibitem{UA2005} A.A. Ungar, {\it Analytic hyperbolic geometry: Mathematical foundations and applications}, World Scientific, Hackensack, 2005.

\bibitem{UA2002} A.A. Ungar,{\it  Beyond the Einstein addition law and its gyroscopic Thomas precession: The theory
of gyrogroups and gyrovector spaces}, Fundamental Theories of Physics, vol. 117, Springer, Netherlands, 2002.

\end{thebibliography}
\end{document}